\documentclass[11pt]{amsart}

\author{Carlo Sanna}
\address{Universit\`a degli Studi di Torino\\Department of Mathematics\\Turin, Italy}
\email{carlo.sanna.dev@gmail.com}
\urladdr{\url{http://orcid.org/0000-0002-2111-7596}}

\keywords{Linear recurrences; greatest common divisor; divisibility}
\subjclass[2010]{Primary: 11B37. Secondary: 11A07, 11B39, 11N25.}
\title{On numbers $n$ relatively prime to the $n$th term of a linear recurrence}

\usepackage{amsmath}
\usepackage{amssymb}
\usepackage{amsthm}
\usepackage{geometry}
\usepackage{color}
\usepackage{hyperref}
\usepackage{enumerate}
\hypersetup{colorlinks=true}

\newtheorem{thm}{Theorem}[section]

\newtheorem{lem}[thm]{Lemma}

\begin{document}

\begin{abstract}
Let $(u_n)_{n \geq 0}$ be a nondegenerate linear recurrence of integers, and let $\mathcal{A}$ be the set of positive integers $n$ such that $u_n$ and $n$ are relatively prime.
We prove that $\mathcal{A}$ has an asymptotic density, and that this density is positive unless $(u_n / n)_{n \geq 1}$ is a linear recurrence.
\end{abstract}

\maketitle

\section{Introduction}

Let $(u_n)_{n \geq 0}$ be a linear recurrence over the integers, that is, $(u_n)_{n \geq 0}$ is a sequence of integers satisfying
\begin{equation*}
u_n = a_1 u_{n - 1} + a_2 u_{n - 2} + \cdots + a_k u_{n - k} ,
\end{equation*}
for all integers $n \geq k$, where $a_1, \ldots, a_k \in \mathbf{Z}$ and $a_k \neq 0$.
To avoid trivialities, we assume that $(u_n)_{n \geq 0}$ is not identically zero.
We refer the reader to \cite[Ch.~1-8]{MR1990179} for the general terminology and theory of linear recurrences.

The set
\begin{equation*}
\mathcal{B}_u := \{n \in \mathbf{N} : n \mid u_n \}
\end{equation*}
has been studied by several researchers.
Assuming that $(u_n)_{n \geq 0}$ is nondegenerate and that its characteristic polynomial has only simple roots, Alba Gonz\'alez, Luca, Pomerance, and Shparlinski~\cite[Theorem~1.1]{MR2928495} proved that
\begin{equation*}
\#\mathcal{B}_u(x) \ll_k \frac{x}{\log x} ,
\end{equation*}
for all sufficiently large $x > 1$.
Andr\'{e}-Jeannin~\cite{MR1131414} and Somer~\cite{MR1271392} studied the arithmetic properties of the elements of $\mathcal{B}_u$ when $(u_n)_{n \geq 0}$ is a Lucas sequence, that is, $(u_0, u_1, k)  = (0, 1, 2)$.
In such a case, generalizing a previous result of Luca and Tron~\cite{MR3409327}, Sanna~\cite{MR3606950} proved the upper bound
\begin{equation*}
\#\mathcal{B}_u(x) \leq x^{1 - \left(\frac1{2} + o(1)\right) \log \log \log x / \log \log x} ,
\end{equation*}
as $x \to +\infty$, where the $o(1)$ depends on $a_1$ and $a_2$.
Furthermore, Corvaja and Zannier~\cite{MR1918678} studied the more general set
\begin{equation*}
\mathcal{B}_{u,v} := \{n \in \mathbf{N} : v_n \mid u_n \} ,
\end{equation*}
where $(v_n)_{n \geq 0}$ is another linear recurrence over the integers.
Under some mild hypotheses on $(u_n)_{n \geq 0}$ and $(v_n)_{n \geq 0}$, they proved that $\mathcal{B}_{u,v}$ has zero asymptotic density~\cite[Corollary~2]{MR1918678}, while Sanna~\cite{San_preprint} gave the bound
\begin{equation*}
\#\mathcal{B}_{u,v}(x) \ll_{u,v} x \cdot \left(\frac{\log \log x}{\log x}\right)^{h_{u,v}} ,
\end{equation*}
for all $x \geq 3$, where $h_{u,v}$ is a positive integer depending only on $(u_n)_{n \geq 0}$ and $(v_n)_{n \geq 0}$.

If $(F_n)_{n \geq 0}$ is the sequence of Fibonacci numbers, Leonetti and Sanna~\cite{LS_preprint} showed that the set
\begin{equation*}
\mathcal{G} := \{\gcd(n, F_n) : n \in \mathbf{N} \}
\end{equation*}
has zero asymptotic density, and that
\begin{equation*}
\#\mathcal{G}(x) \gg \frac{x}{\log x} ,
\end{equation*}
for all $x \geq 2$.

On the other hand, the set
\begin{equation*}
\mathcal{A}_u = \{n \in \mathbf{N} : \gcd(n, u_n) = 1\}
\end{equation*}
does not seem to have attracted so much attention.
We prove the following result:

\begin{thm}\label{thm:main}
For any nondegenerate linear recurrence $(u_n)_{n \geq 0}$, the asymptotic density $\mathbf{d}(\mathcal{A}_u)$ of $\mathcal{A}_u$ exists.
Moreover, if $(u_n / n)_{n \geq 1}$ is not a linear recurrence then $\mathbf{d}(\mathcal{A}_u) > 0$.
Otherwise, $\mathcal{A}_u$ is finite and, a fortiori, $\mathbf{d}(\mathcal{A}_u) = 0$.
\end{thm}

We remark that given the initial conditions and the coefficients of a linear recurrence $(u_n)_{n \geq 0}$, it is easy to test effectively if $(u_n / n)_{n \geq 1}$ is a linear recurrence or not (see Lemma~\ref{lem:unovern}, in \S\ref{sec:preliminaries}).

\subsection*{Notation}

Throughout, the letter $p$ always denotes a prime number.
For a set of positive integers $\mathcal{S}$, we put $\mathcal{S}(x):=\mathcal{S}\cap [1,x]$ for all $x\ge 1$, and we recall that the asymptotic density $\mathbf{d}(\mathcal{S})$ of $\mathcal{S}$ is defined as the limit of the ratio $\#\mathcal{S}(x) / x$ as $x \to +\infty$, whenever this exists.
We employ the Landau--Bachmann ``Big Oh'' and ``little oh'' notations $O$ and $o$, as well as the associated Vinogradov symbols $\ll$ and $\gg$, with their usual meanings.
Any dependence of the implied constants is explicitly stated or indicated with subscripts.

\section{Preliminaries}\label{sec:preliminaries}

In this section we give some definitions and collect some preliminary results needed in the later proofs.
Let $f_u$ be the characteristic polynomial of $(u_n)_{n \geq 0}$, i.e.,
\begin{equation*}
f_u(X) = X^k - a_1 X^{k - 1} - a_2 X^{k - 2} - \cdots - a_k .
\end{equation*}
Moreover, let $\mathbf{K}$ be the splitting field of $f_u$ over $\mathbf{Q}$, and let $\alpha_1, \ldots, \alpha_r \in \mathbf{K}$ be all the distinct roots of $f_u$.
It is well known that there exist $g_1, \ldots, g_r \in \mathbf{K}[X]$ such that
\begin{equation}\label{equ:genpowsum}
u_n = \sum_{i = 1}^r g_i(n)\;\! \alpha_i^n ,
\end{equation}
for all integers $n \geq 0$.
We define $B_u$ as the smallest positive integer such that all the coefficients of the polynomials $B_u g_1, \ldots, B_u g_r$ are algebraic integers.

We have the following easy lemma.

\begin{lem}\label{lem:unovern}
$\{u_n / n\}_{n \geq 1}$ is a linear recurrence if and only if 
\begin{equation}\label{equ:gizero}
g_1(0) = \cdots = g_r(0) = 0 .
\end{equation}
In such a case, $\mathcal{A}_u$ is finite.
\end{lem}
\begin{proof}
The first part of the lemma follows immediately from the fact that any linear recurrence can be written as a generalized power sum like (\ref{equ:genpowsum}) in a unique way (assuming the roots $\alpha_1, \ldots, \alpha_r$ are distinct, and up to the order of the addends).
For the second part, if (\ref{equ:gizero}) holds then for all positive integer $n$ we have that
\begin{equation*}
\frac{B_u u_n}{n} = \sum_{i = 1}^r \frac{B_u g_i(n)}{n} \, \alpha_i^n 
\end{equation*}
is both a rational number and an algebraic integer, hence it is an integer.
Therefore, $n \mid B_u u_n$, and so $\gcd(n, u_n) = 1$ only if $n \mid B_u$, which in turn implies that $\mathcal{A}_u$ is finite.
\end{proof}

For the rest of this section, we assume that $(u_n)_{n \geq 0}$ is nondegenerate and that $f_u$ has only simple roots, hence, in particular, $r = k$.
We write $\Delta_u$ for the discriminant of the polynomial $f_u$, and we recall that $\Delta_u$ is a nonzero integer.
If $k \geq 2$, then for all integers $x_1, \ldots, x_k$ we set
\begin{equation*}
D_u(x_1, \ldots, x_k) := \det(\alpha_i^{x_j})_{1 \leq i, j \leq k} ,
\end{equation*}
and for any prime number $p$ not dividing $a_k$ we define $T_u(p)$ as the greatest integer $T \geq 0$ such that $p$ does not divide
\begin{equation*}
\prod_{1 \leq x_2, \ldots, x_k \leq T} \max\!\left\{1, |N_\mathbf{K} (D_u(0, x_2, \ldots, x_k))| \right\} ,
\end{equation*}
where $N_\mathbf{K}(\alpha)$ denotes the norm of $\alpha \in \mathbf{K}$ over $\mathbf{Q}$, and the empty product is equal to $1$.
It is known that such $T$ exists \cite[p.~88]{MR1990179}.
If $k = 1$, then we set $T_u(p) := +\infty$ for all prime numbers $p$ not dividing $a_1$.
Note that $T_u(p) = 0$ if and only if $k = 2$ and $p$ divides $\Delta_u$.

Finally, for all $\gamma \in {]0,1[}$, we define
\begin{equation*}
\mathcal{P}_{u,\gamma} := \{p : p \nmid a_k, \; T_u(p) < p^\gamma \} .
\end{equation*}

We are ready to state two important lemmas regarding $T_u(p)$~\cite[Lemma~2.1, Lemma~2.2]{MR2928495}.

\begin{lem}\label{lem:Pugamma}
For all $\gamma \in {]0,1[}$ and $x \geq 2^{1/ \gamma}$ we have
\begin{equation*}
\#\mathcal{P}_{u,\gamma}(x) \ll_u \frac{x^{k\gamma}}{\gamma \log x} .
\end{equation*}
\end{lem}

\begin{lem}\label{lem:upmcong}
Assume that $p$ is a prime number not dividing $a_k B_u \Delta_u$ and relatively prime with at least one term of $(u_n)_{n \geq 0}$.
Then, for all $x \geq 1$, the number of positive integers $m \leq x$ such that $u_{pm} \equiv 0 \pmod p$ is
\begin{equation*}
O_k\!\left(\frac{x}{T_u(p)} + 1\right) .
\end{equation*}
\end{lem}

Actually, in~\cite{MR2928495} both Lemma~\ref{lem:Pugamma} and Lemma~\ref{lem:upmcong} were proved only for $k \geq 2$.
However, one can easily check that they are true also for $k = 1$.

\section{Proof of Theorem~\ref{thm:main}}

For all integers $n \geq 0$, define
\begin{equation*}
v_n := B_u \sum_{i = 1}^r \frac{g_i(n) - g_i(0)}{n} \;\! \alpha_i^n \quad\text{and}\quad w_n := B_u \sum_{i = 1}^r g_i(0) \;\! \alpha_i^n .
\end{equation*}
Note that both $(v_n)_{n \geq 0}$ and $(w_n)_{n \geq 0}$ are linear recurrences of algebraic integers, and that the characteristic polynomial of $(w_n)_{n \geq 0}$ has only simple roots.

Let $\mathcal{G}$ be the Galois group of $\mathbf{K}$ over $\mathbf{Q}$.
Since $u_n$ is an integer, for any $\sigma \in \mathcal{G}$ we have that
\begin{align}\label{equ:nvnwn}
nv_n + w_n = B_u u_n = \sigma(B_u u_n) = \sigma(n v_n + w_n) = n\sigma(v_n) + \sigma(w_n) ,
\end{align}
for all integers $n \geq 0$.
In (\ref{equ:nvnwn}) note that both $n\sigma(v_n)$ and $\sigma(w_n)$ are linear recurrences, and the first is a multiple of $n$, while the characteristic polynomial of the second has only simple roots. 
Since the expression of a linear recurrence as a generalized power sum is unique, from (\ref{equ:nvnwn}) we get that $w_n = \sigma(w_n)$ for any $\sigma \in \mathcal{G}$, hence $w_n$ is an integer.

Thanks to Lemma~\ref{lem:unovern}, we know that $(w_n)_{n \geq 0}$ is identically zero if and only if $(u_n / n)_{n \geq 1}$ is a linear recurrence, and in such a case $\mathcal{A}_u$ is finite, so that the claim of Theorem~\ref{thm:main} is obvious.
Hence, we assume that $(w_n)_{n \geq 0}$ is not identically zero.

For the sake of convenience, put $\mathcal{C}_u := \mathbf{N} \setminus \mathcal{A}_u$.
Thus we have to prove that the asymptotic density of $\mathcal{C}_u$ exists and is less than $1$.
For each $y > 0$, we split $\mathcal{C}_u$ into two subsets:
\begin{align*}
\mathcal{C}_{u,y}^- &:= \{n \in \mathcal{C}_u : p \mid \gcd(n, u_n) \text{ for some } p \leq y\} , \\
\mathcal{C}_{u,y}^+ &:= \mathcal{C}_u \setminus \mathcal{C}_{u,y}^- .
\end{align*}
It is well known that $(u_n)_{n \geq 0}$ is definitively periodic modulo $p$, for any prime number $p$.
Therefore, it is easy to see that $\mathcal{C}_{u,y}^-$ is an union of finitely many arithmetic progressions and a finite subset of $\mathbf{N}$.
In particular, $\mathcal{C}_{u,y}^-$ has an asymptotic density.
If we put $\delta_y := \mathbf{d}(\mathcal{C}_{u,y}^-)$, then it is clear that $\delta_y$ is a bounded nondecreasing function of $y$, hence the limit 
\begin{equation}\label{equ:limdelta}
\delta := \lim_{y \to +\infty} \delta_y
\end{equation}
exists finite.
We shall prove that $\mathcal{C}_u$ has asymptotic density $\delta$.
Hereafter, all the implied constants may depend on $(u_n)_{n \geq 0}$ and $k$.
If $n \in \mathcal{C}_{u, y}^+(x)$ then there exists a prime $p > y$ such that $p \mid n$ and $p \mid u_n$.
Furthermore, $B_u u_n = n v_n + w_n$ implies that $p \mid w_n$.
Hence, we can write $n = pm$ for some positive integer $m \leq x / p$ such that $w_{pm} \equiv 0 \pmod p$.
For sufficiently large $y$, we have that $p$ does not divide $f_w(0) B_w \Delta_w$ and is coprime with at least one term of $(w_s)_{s \geq 0}$, since $(w_s)_{s \geq 0}$ is not identically zero.

Therefore, by applying Lemma~\ref{lem:upmcong} to $(w_s)_{s \geq 0}$, we get that the number of possible values of $m$ is at most
\begin{equation*}
O\!\left(\frac{x}{pT_w(p)} + 1\right) .
\end{equation*}
As a consequence,
\begin{equation}\label{equ:bound1}
\#\mathcal{C}_{u, y}^+(x) \ll \sum_{y < p \leq x} \left(\frac{x}{pT_w(p)} + 1\right) \ll x \cdot \left(\sum_{p > y} \frac1{pT_w(p)} + \frac1{\log x}\right) ,
\end{equation}
where we also used the Chebyshev's bound for the number of primes not exceeding $x$.
Setting $\gamma := 1/(k + 1)$, by partial summation and Lemma~\ref{lem:Pugamma}, we have
\begin{equation}\label{equ:bound2}
\sum_{\substack{p > y \\ p \in \mathcal{P}_{w,\gamma}}} \frac1{pT_w(p)} \leq \sum_{\substack{p > y \\ p \in \mathcal{P}_{w,\gamma}}} \frac1{p} = \left[\frac{\#\mathcal{P}_{w,\gamma}(t)}{t}\right]_{t = y}^{+\infty} + \int_y^{+\infty} \frac{\#\mathcal{P}_{w,\gamma}(t)}{t^2}\mathrm{d}t \ll \frac1{y^{1 - k\gamma}} = \frac1{y^\gamma}.
\end{equation}
On the other hand, 
\begin{equation}\label{equ:bound3}
\sum_{\substack{p > y \\ p \notin \mathcal{P}_{w,\gamma}}} \frac1{pT_w(p)} \leq \sum_{\substack{p > y \\ p \notin \mathcal{P}_{w,\gamma}}} \frac1{p^{1 + \gamma}} \ll \int_y^{+\infty} \frac{\mathrm{d}t}{t^{1 + \gamma}} \ll \frac1{y^{\gamma}}
\end{equation}
Thus, putting together (\ref{equ:bound1}), (\ref{equ:bound2}), and (\ref{equ:bound3}), we obtain
\begin{equation*}
\frac{\#\mathcal{C}_{u, y}^+(x)}{x} \ll \frac1{y^\gamma} + \frac1{\log x} ,
\end{equation*}
so that
\begin{equation}\label{equ:limsup}
\limsup_{x \to +\infty} \left|\frac{\#\mathcal{C}_u(x)}{x} - \delta_y\right| = \limsup_{x \to +\infty} \left|\frac{\#\mathcal{C}_u(x)}{x} - \frac{\#\mathcal{C}_{u,y}^-(x)}{x}\right| = \limsup_{x \to +\infty} \frac{\#\mathcal{C}_{u,y}^+(x)}{x} \ll \frac1{y^\gamma} ,
\end{equation}
hence, by letting $y \to +\infty$ in (\ref{equ:limsup}) and by using (\ref{equ:limdelta}), we get that $\mathcal{C}_u$ has asymptotic density $\delta$.

It remains only to prove that $\delta < 1$.
Clearly,
\begin{equation*}
\mathcal{C}_{u,y}^- \subseteq \{n \in \mathbf{N} : p \mid n \text{ for some } p \leq y\} ,
\end{equation*}
so that, by Eratosthenes' sieve and Mertens' third theorem~\cite[Ch.~I.1, Theorem~11]{Ten95}, we have
\begin{equation}\label{equ:limsupCminus}
\limsup_{x \to +\infty} \frac{\#\mathcal{C}_{u,y}^-(x)}{x} \leq 1 - \prod_{p \leq y}\left(1 - \frac1{p}\right) \leq 1 - \frac{c_1}{\log y} ,
\end{equation}
for all $y \geq 2$, where $c_1 > 0$ is an absolute constant.
Hence, putting together (\ref{equ:limsupCminus}) and the last part of (\ref{equ:limsup}), we get
\begin{equation}\label{equ:lastlimsup}
\delta = \lim_{x \to +\infty} \frac{\#\mathcal{C}_u(x)}{x} \leq \limsup_{x \to +\infty} \frac{\#\mathcal{C}_{u,y}^-(x)}{x} + \limsup_{x \to +\infty} \frac{\#\mathcal{C}_{u,y}^+(x)}{x} \leq 1 - \left(\frac{c_1}{\log y} - \frac{c_2}{y^{\gamma}}\right) ,
\end{equation}
for all $y \geq 2$, where $c_2 > 0$ is an absolute constant.

Finally, picking a sufficiently large $y$, depending on $c_1$ and $c_2$, the bound (\ref{equ:lastlimsup}) yields $\delta < 1$.
The proof of Theorem~\ref{thm:main} is complete.

\bibliographystyle{amsplain}

\end{document}